\documentclass[leqno,12pt]{article}

\usepackage{amssymb,amscd}

\usepackage{amsthm}

\theoremstyle{plain} 

\newtheorem{theorem}{\sc Theorem}[section]

\newtheorem{lemma}[theorem]{\sc Lemma}

\newtheorem{corollary}[theorem]{\sc Corollary}

\newtheorem{proposition}[theorem]{\sc Proposition}

\theoremstyle{definition}

\newtheorem{remark}[theorem]{\sc Remark}

\newtheorem{remarks}[theorem]{\sc Remarks}

\newcommand\bA{{\mathbb A}}
\newcommand\bC{{\mathbb C}}

\newcommand\bG{{\mathbb G}}

\newcommand\bZ{{\mathbb Z}}

\newcommand\bmu{{\mathbb \mu}}

\newcommand\cL{{\mathcal L}}

\newcommand\cO{{\mathcal O}}

\newcommand\aff{{\rm aff}}

\newcommand\gp{{\rm gp}}
\newcommand\id{{\rm id}}

\newcommand\red{{\rm red}}

\newcommand\Aut{{\rm Aut}}
\newcommand\End{{\rm End}}
\newcommand\GL{{\rm GL}}
\newcommand\Hom{{\rm Hom}}
\newcommand\Lie{{\rm Lie}}

\newcommand\Spec{{\rm Spec}}

\makeatletter

\def\address#1#2{\begingroup
\noindent\parbox[t]{7.8cm}{%
\small{\scshape\ignorespaces#1}\par\vskip1ex
\noindent\small{\itshape E-mail address}%
\/: #2\par\vskip4ex}\hfill%
\endgroup}

\makeatother

\title{On automorphism groups of fiber bundles}

\author{Michel Brion}

\date{}

\begin{document}

\maketitle


\begin{abstract}
We obtain analogues of classical results on automorphism groups 
of holomorphic fiber bundles, in the setting of group schemes. 
Also, we establish a lifting property of the connected automorphism group, 
for torsors under abelian varieties. These results will be applied to the 
study of homogeneous bundles over abelian varieties.
\end{abstract}

\section{Introduction}
\label{sec:introduction}

This work arose from a study of homogeneous bundles over an abelian 
variety $A$, that is, of those principal bundles with base $A$ and fiber  
an algebraic group $G$, that are isomorphic to all of their pull-backs 
by the translations of $A$ (see \cite{Br2}). In the process of that 
study, it became necessary to obtain algebro-geometric analogues of 
two classical results about automorphisms of fiber bundles in complex 
geometry. The first one, due to Morimoto (see \cite{Mo}), asserts that 
the equivariant automorphism group of a principal bundle over 
a compact complex manifold, with fiber a complex Lie group, 
is a complex Lie group as well. The second one, a result of Blanchard 
(see \cite{Bl}), states that a holomorphic action of a complex 
connected Lie group on the total space of a locally trivial fiber 
bundle of complex manifolds descends to a holomorphic action 
on the base, provided that the fiber is compact and connected.

Also, we needed to show the existence in the category of schemes
of certain fiber bundles associated to a $G$-torsor (or principal bundles)
$\pi: X \to Y$, where $G$ is a connected group scheme and $X, Y$ are algebraic 
schemes; namely, those fiber bundles  $X \times^G Z \to Y$ associated
to $G$-homogeneous varieties $Z$. Note that the fiber bundle associated
to an arbitrary $G$-scheme $Z$ exists in the category of algebraic 
spaces, but may fail to be a scheme (see \cite{Bi,KM}).

Finally, we were led to a lifting result which reduces the study of 
homogeneous bundles to the case that the structure group is linear,
and does not seem to have its holomorphic counterpart. It asserts 
that given a $G$-torsor $\pi : X \to Y$ where $G$ is an abelian variety 
and $X,Y$ are smooth complete algebraic varieties, the connected 
automorphism group of $X$ maps onto that of $Y$ under the homomorphism 
provided by the analogue of Blanchard's theorem.

In this paper, we present these preliminary results which may have
independent interest, with (hopefully) modest prerequisites.
Section 2 is devoted to a scheme-theoretic version of Blanchard's 
theorem: a proper morphism of schemes $\pi: X \to Y$ such that 
$\pi_*(\cO_X) = \cO_Y$ induces a homomorphism 
$\pi_* : \Aut^o(X) \to \Aut^o(Y)$ between the neutral components of 
the automorphism group schemes (Corollary \ref{cor:dir}). Our proof
is an adaptation of that given in \cite{Ak} in the setting of 
complex spaces. 

In Section 3, we consider a torsor $\pi : X \to Y$ under a connected
group scheme $G$, and show the existence of the associated 
fiber bundle $X \times^G G/H = X/H$ for any subgroup scheme 
$H \subset G$ (Theorem \ref{thm:fact}). As a consequence, 
$X \times^G Z$ exists when $Z$ is the total space of a $G$-torsor, 
or a group scheme where $G$ acts via a homomorphism (Corollary 
\ref{cor:join}). Another application of Theorem \ref{thm:fact} 
concerns the quasi-projectivity of torsors (Corollary \ref{cor:fact}); 
it builds on work of Raynaud, who showed e.g. the local 
quasi-projectivity of homogeneous spaces over a normal scheme 
(see \cite{Ra}).
 
The automorphism groups of torsors are studied in Section 4.
In particular, we obtain a version of Morimoto's theorem:
the equivariant automorphisms of a torsor over a proper scheme
form a group scheme, locally of finite type (Theorem 
\ref{thm:fin}). Here our proof, based on an equivariant completion
of the structure group, is quite different from the original one.
We also analyze the relative equivariant automorphism group of 
such a torsor; this yields a version of Chevalley's structure theorem 
for algebraic groups in that setting (Proposition \ref{prop:autrel}).

The final Section 5 contains a full description of relative 
equivariant automorphisms for torsors under abelian varieties 
(Proposition \ref{prop:alb}) and our lifting result for automorphisms 
of the base (Theorem \ref{thm:abel}).

\bigskip 

\noindent
{\bf Acknowledgements.} Many thanks to Ga\"el R\'emond for several
clarifying discussions, and special thanks to the referee for very 
helpful comments and corrections. In fact, the final step of the 
proof of Theorem \ref{thm:fact} is taken from the referee's report; 
the end of the proof of Corollary \ref{cor:dir}, and the proof of 
Corollary \ref{cor:join} (ii), closely follow his/her suggestions.

\bigskip 

\noindent
{\bf Notation and conventions.}
Throughout this article, we consider algebraic varieties, schemes,
and morphisms over an algebraically closed field $k$. Unless 
explicitly mentioned, we will assume that the considered schemes 
are of finite type over $k$ (such schemes are also called algebraic 
schemes). By a point of a scheme $X$, we will mean a closed point
unless explicitly mentioned.  
A \emph{variety} is an integral separated scheme.

We will use \cite{DG} as a general reference for group schemes.
Given such a group scheme $G$, we denote by $\mu_G : G \times G \to G$
the multiplication and by $e_G \in G(k)$ the neutral element. 
The neutral component of $G$ is denoted by $G^o$, and the Lie algebra 
by $\Lie(G)$. 

We recall that an \emph{action} of $G$ on a scheme $X$ is a morphism
$$
\alpha : G \times X \longrightarrow X, \quad 
(g,x) \longmapsto g \cdot x
$$
such that the composite map
$$
\CD
X @>{e_G \times  \id_X}>> G \times X @>{\alpha}>> X
\endCD
$$
is the identity, and the square
$$
\CD
G \times G \times X @>{\id_G \times \alpha}>> G \times X \\
@V{\mu_G \times \id_X}VV @V{\alpha}VV \\
G \times X @>{\alpha}>> X \\
\endCD
$$
commutes. We then say that $X$ is a $G$-\emph{scheme}.
A morphism $f : X \to Y$ between two $G$-schemes is 
called \emph{equivariant} if the square
$$
\CD
G \times X @>{\alpha}>> X \\
@V{\id_G \times f}VV @V{f}VV \\
G \times Y @>{\beta}>> Y\\
\endCD
$$
commutes (with the obvious notation). We then say that 
$f$ is a $G$-\emph{morphism}.

A smooth group scheme will be called an algebraic group.
By Chevalley's structure theorem (see \cite[Theorem 16]{Ro}, 
or \cite{Co} for a modern proof), every connected algebraic 
group $G$ has a largest closed connected normal affine subgroup 
$G_{\aff}$; moreover, the quotient $G/G_{\aff} =: A(G)$ is an abelian 
variety. This yields an exact sequence of connected algebraic groups
$$
1 \longrightarrow  G_{\aff}  \longrightarrow  G 
\longrightarrow A(G) \longrightarrow 1.
$$

\section{Descending automorphisms for fiber spaces}
\label{sec:desc}

We begin with the following scheme-theoretic version of
a result of Blanchard (see \cite[Section I.1]{Bl} and also 
\cite[Lemma 2.4.2]{Ak}).

\begin{proposition}\label{prop:blan}
Let $G$ be a connected group scheme, $X$ a $G$-scheme,
$Y$ a scheme, and $\pi : X \to Y$ a proper morphism such that 
$\pi_*(\cO_X) = \cO_Y$. Then there is a unique $G$-action on $Y$ 
such that $\pi$ is equivariant.
\end{proposition}

\begin{proof}
We will consider a scheme $Z$ as the ringed space $(Z(k), \cO_Z)$ 
where the set $Z(k)$ is equipped with the Zariski topology; this makes 
sense as $Z$ is of finite type. 
  
 We first claim that the abstract group $G(k)$ permutes
the fibers of $\pi : X(k) \to Y(k)$ (note that these fibers are non-empty 
and connected, since $\pi_*(\cO_X) = \cO_Y$). Let $y \in Y(k)$ and denote 
by $F_y$ the set-theoretic fiber of $\pi$ at $y$, viewed as a closed
reduced subscheme of $X$. Then the map
$$
\varphi : G_{\red} \times F_y \longrightarrow Y, \quad 
(g,x) \longmapsto \pi(g\cdot x)
$$
maps $\{ e_G \} \times F_y$ to the point $y$. Moreover, $G_{\red}$ 
is a variety, and $F_y$ is connected and proper. By the rigidity lemma 
(see \cite[p.~43]{Mu}), it follows that $\varphi$ maps 
$\{ g \} \times F_y$ to a point for any $g \in G(k)$, i.e., 
$g \cdot F_y \subset F_{g \cdot y}$. Thus, 
$g^{-1} \cdot F_{g \cdot y} \subset F_y$ and hence 
$g \cdot F_y = F_{g \cdot y}$. This implies our claim.

That claim yields a commutative square
$$
\CD
G(k) \times X(k) @>{\alpha}>> X(k) \\
@V{\id_G \times  \pi}VV @V{\pi}VV \\
G(k) \times Y(k) @>{\beta}>> Y(k),\\
\endCD
$$
where $\beta$ is an action of the (abstract) group $G(k)$.

Next, we show that $\beta$ is continuous. It suffices to
show that $\beta^{-1}(Z)$ is closed for any closed subset 
$Z \subset Y(k)$. But 
$(\id_G, \pi)^{-1} \beta^{-1}(Z) =\alpha^{-1} \pi^{-1}(Z)$
is closed, and $(\id_G,\pi)$ is proper and surjective; this yields 
our assertion.

Finally, we define a morphism of sheaves of $k$-algebras
$$
\beta^{\#} : \cO_Y \longrightarrow \beta_*(\cO_{G \times Y}).
$$
For this, to any open subset $V \subset Y$, we associate
a homomorphism of algebras
$$
\beta^{\#}(V) : \cO_Y(V) \longrightarrow 
\cO_{G \times Y} \big( \beta^{-1}(V) \big).
$$
By assumption, the left-hand side is isomorphic to 
$\cO_X \big( \pi^{-1}(V) \big)$, and the right-hand side to 
$$
\cO_{G \times X}\big( (\id_G,\pi)^{-1}\beta^{-1}(V) \big) =
\cO_{G \times X}\big( \alpha^{-1}\pi^{-1}(V) \big).
$$
We define
$\beta^{\#}(V) := \alpha^{\#} \big( \pi^{-1}(V) \big)$.
Now it is straightforward to verify that 
$(\beta,\beta^{\#})$ is a morphism of locally ringed spaces;
this yields a morphism of schemes $\beta : G \times Y \to Y$.
By construction, $\beta$ is the unique morphism such that 
the square 
$$
\CD
G \times X @>{\alpha}>> X \\
@V{\id_G \times \pi}VV @V{\pi}VV \\
G \times Y @>{\beta}>> Y\\
\endCD
$$
commutes. 

It remains to show that $\beta$ is an action of the
group scheme $G$. Note that $e_G$ acts on $X(k)$ via 
the identity; moreover, the composite morphism of sheaves
$$
\CD
\cO_Y @>{\beta^{\#}}>> \beta_*(\cO_{G \times Y}) 
@>{(e_G \times \id_Y)^{\#}}>> \beta_*( \cO_{\{e_G\} \times Y})  
\cong \cO_Y
\endCD
$$ 
is the identity, since so is the analogous morphism 
$$
\CD
\cO_X @>{\alpha^{\#}}>> \alpha_*(\cO_{G \times X}) 
@>{(e_G \times \id_X)^{\#}}>> \alpha_* (\cO_{\{e_G\} \times X}) 
\cong \cO_X
\endCD
$$
and $\pi_*(\cO_X) = \cO_Y$. Likewise, the square
$$
\CD
G \times G \times Y @>{\id_G \times \beta}>> G \times Y \\
@V{\mu_G \times \id_Y}VV @V{\beta}VV \\
G \times Y @>{\beta}>> Y \\
\endCD
$$
commutes on closed points, and the corresponding square 
of morphisms of sheaves commutes as well, since the analogous
square with $Y$ replaced by $X$ commutes.
\end{proof}

This proposition will imply a result of descent for group scheme
actions, analogous to \cite[Proposition I.1]{Bl} (see also 
\cite[Proposition 2.4.1]{Ak}). To state that result, we need 
some recollections on automorphism functors.
 
Given a scheme $S$, we denote by $\Aut_S(X \times S)$ the group of 
automorphisms of $X \times S$ viewed as a scheme over $S$. The assignement 
$S \mapsto \Aut_S(X \times S)$ yields a group functor $Aut(X)$, i.e., 
a contravariant functor from the category of schemes to that of groups. 
If $X$ is proper, then $Aut(X)$ is represented by a group scheme $\Aut(X)$, 
locally of finite type (see \cite[Theorem 3.7]{MO}). In particular, 
the neutral component $\Aut^o(X)$ is a group scheme of finite type.
Also, recall that
\begin{equation}\label{eqn:lie}
\Lie \, \Aut(X) \cong \Gamma(X,T_X)
\end{equation}
where the right-hand side denotes the Lie algebra of global vector
fields on $X$, that is, of derivations of $\cO_X$.

We now are in a position to state:

\begin{corollary}\label{cor:dir}
Let $\pi : X \to Y$ be a morphism of proper schemes such that 
$\pi_* (\cO_X) = \cO_Y$. Then $\pi$ induces a homomorphism of group 
schemes 
$$
\pi_* : \Aut^o(X) \longrightarrow \Aut^o(Y).
$$ 
\end{corollary}

\begin{proof}
This is a formal consequence of Proposition \ref{prop:blan}.
Specifically, let $G := \Aut^o(X)$ and consider the 
$G$-action on $Y$ obtained in that proposition. This yields 
a automorphism of $Y \times G$ as a scheme over $G$,
$$
(y,g) \longmapsto (g \cdot y, g),
$$
and in turn a morphism (of schemes) 
$$
\pi_* : G \longrightarrow \Aut(Y).
$$
Moreover, $\pi_*(e_G) = e_{\Aut(Y)}$ since $e_G$ acts 
via the identity. As $G$ is connected, it follows that
the image of $\pi_*$ is contained in $\Aut^o(Y) =: H$.
In other words, we have a morphism of schemes
$\pi_* : G \to H$ such that $\pi_*(e_G) = e_H$. It remains
to check that $\pi_*$ is a homomorphism; but this follows
from the fact that $\pi_*$ corresponds to the $G$-action
on $Y$, and hence yields a morphism of group functors.
\end{proof}

Given two complete varieties $X$ and $Y$, the preceding corollary
applies to the projections 
$$
p : X \times Y \to X, \quad q : X \times Y \to Y
$$ 
and yields homomorphisms 
$$
p_* : \Aut^o(X) \times \Aut^o(Y) \to \Aut^o(X), \quad 
q_* : \Aut^o(X) \times \Aut^o(Y) \to \Aut^o(Y).
$$
This implies readily the following analogue of 
\cite[Corollaire, p. 161]{Bl}:

\begin{corollary}\label{cor:prod}
Let $X$ and $Y$ be complete varieties. Then the homomorphism
$$
(p_*,q_*) : \Aut^o(X \times Y) \longrightarrow 
\Aut^o(X) \times \Aut^o(Y)
$$
is an isomorphism, with inverse the natural homomorphism 
$$
\Aut^o(X) \times \Aut^o(Y) \longrightarrow \Aut^o(X \times Y),
\quad (g, h) \longmapsto \big( (x,y) \mapsto (g(x),h(y)  \big).
$$
\end{corollary}

More generally, the isomorphism 
$$
\Aut^o(X \times Y) \cong \Aut^o(X) \times \Aut^o(Y)
$$
holds for those proper schemes $X$ and $Y$ such that
$\cO(X) = \cO(Y) = k$, but may fail for arbitrary proper schemes.
Indeed, let $X$ be a complete variety having non-zero 
global vector fields, and let $Y := \Spec \, k[\varepsilon]$
where $\epsilon^2 = 0$; denote by $y$ the closed point of $Y$. 
Then we have an exact sequence
$$
1 \longrightarrow \Gamma(X,T_X) \longrightarrow \Aut_Y(X \times Y)
\longrightarrow \Aut(X) \longrightarrow 1,
$$
where the map on the right is obtained by restricting to
$X \times \{y\}$. This identifies the vector group $\Gamma(X,T_X)$ 
to a closed subgroup of $\Aut^o(X \times Y)$, which is not
in the image of the natural homomorphism.

Likewise, $\Aut(X \times Y)$ is generally strictly larger than
$\Aut(X) \times \Aut(Y)$ (e.g. take $Y = X$ and consider the 
automorphism $(x,y) \mapsto (y,x)$).

\section{Torsors and asssociated fiber bundles}
\label{sec:ass}

Consider a group scheme $G$, a $G$-scheme $X$, and a $G$-invariant 
morphism 
\begin{equation}\label{eqn:tor}
\pi: X \longrightarrow Y,
\end{equation}
where $Y$ is a scheme. We say that $X$ is a 
{\it $G$-torsor over $Y$}, if $\pi$ is faithfully flat 
and the morphism
\begin{equation}\label{eqn:act}
\alpha \times p_2 : G \times X \longrightarrow X \times_Y X, \quad 
(g,x) \longmapsto (g \cdot x, x)
\end{equation}
is an isomorphism. The latter condition is equivalent to the 
existence of a faithfully flat morphism $f: Y' \to Y$ such that the 
pull-back torsor $\pi' : X \times_Y Y' \to Y'$ is trivial.
(Since our schemes are assumed to be of finite type, $\pi$ is 
quasi-compact and finitely presented; thus, there is no need 
to distinguish between the fppf and the fpqc topology).

For a $G$-torsor (\ref{eqn:tor}), the morphism $\pi$ is surjective, 
and its geometric fiber $X_{\bar{y}}$ is isomorphic to $G_{\bar{y}}$ 
for any (possibly non-closed) point $y \in Y$. In particular, 
$\pi$ is smooth if and only if $G$ is an algebraic group; under 
that assumption, $X$ is smooth (resp. normal) if and only if 
so is $Y$. 

Also, note that $\pi$ is a universal geometric quotient in the 
sense of \cite[Section 0]{MFK}, and hence a universal categorical
quotient (see [loc. cit., Proposition 0.1]). In particular, 
$Y(k) = X(k)/G(k)$ and $\cO_Y = \pi_*(\cO_X)^G$ 
(the subsheaf of $G$-invariants in $\pi_*(\cO_X)$). Thus, we will
also denote $Y$ by $X/G$.

\begin{remark}
If $G$ is an affine algebraic group, then every $G$-torsor 
(\ref{eqn:tor}) is \emph{locally isotrivial}, i.e., for any point
$y \in Y$ there exist an open subscheme $V \subset Y$ containing
$y$ and a finite \'etale surjective morphism $f : V' \to V$
such that the pull-back torsor $X \times_V V'$ is trivial
(this result is due to Grothendieck, see \cite[Lemme XIV 1.4]{Ra} 
for a detailed proof). The local isotriviality of $\pi$ also holds 
if $G$ is an algebraic group and $Y_{\red}$ is normal, as a consequence 
of [loc. cit., Th\'eor\`eme XIV 1.2]. In particular, $\pi$ is
locally trivial for the \'etale topology in both cases.

Yet there exist torsors under algebraic groups that are not 
locally isotrivial, see [loc. cit., XIII 3.1] (reformulated in more 
concrete terms in \cite[Example 6.2]{Br1}) for an example where 
$Y$ is a rational nodal curve, and $G$ is an abelian variety having 
a point of infinite order.
\end{remark}

Given a $G$-torsor (\ref{eqn:tor}) and a $G$-scheme $Z$, we may
view $X \times Z$ as a $G$-scheme for the diagonal action, and
ask if there exist a $G$-torsor $\varpi: X \times Z \to W$ 
where $W$ is a scheme, and a morphism $q : W \to Y$ such that 
the square
$$
\CD
X \times Z @>{p_1}>> X \\
@V{\varpi}VV @V{\pi}VV \\
W @>{q}>> Y \\
\endCD
$$
is cartesian; here $p_1$ denotes the first projection. 
Then $q$ is called the 
{\it associated fiber bundle with fiber $Z$}.
The quotient scheme $W$ will be denoted by $X \times^G Z$.

The answer to this question is positive if $Z$ admits an ample 
$G$-linearized invertible sheaf (as follows from descent theory; 
see \cite[Proposition 7.8]{SGA1} and also 
\cite[Proposition 7.1]{MFK}). In particular, the answer is
positive if $Z$ is affine. Yet the answer is generally negative,
even if $Z$ is a smooth variety; see \cite{Bi}. However, 
associated fiber bundles do exist in the category of algebraic spaces, 
see \cite[Corollary 1.2]{KM}.

Of special interest is the case that the fiber is a group scheme $G'$
where $G$ acts through a homomorphism $f: G \to G'$. Then 
$X' := X \times^G G'$ is a $G'$-torsor over $Y$, obtained from $X$ 
by \emph{extension of the structure group}. If $f$ identifies
$G$ with a closed subgroup scheme of $G'$, then $X'$ comes with a
$G'$-morphism to $G'/G$ arising from the projection 
$X \times G' \to G'$. Conversely, the existence of such a morphism
yields a reduction of structure group, in view of the following 
standard result:

\begin{lemma}\label{lem:ind}
Let $G$ be a group scheme, $H$ a subgroup scheme, and $X$ a $G$-scheme
equipped with a $G$-morphism $f : X \to G/H$. Denote by $Z$ the fiber of
$f$ at the base point of $G/H$, so that $Z$ is an $H$-scheme.

Then $f$ is faithfully flat, and the natural map 
$G \times Z \to X$ factors through a $G$-isomorphism
$G \times^H Z \cong X$. 

If $\pi : X \to Y$ is a $G$-torsor, then the restriction 
$\pi_{\vert Z} : Z \to Y$ is an $H$-torsor.
\end{lemma}

\begin{proof}
Form and label the cartesian square
$$
\CD
X' @>{f'}>> G \\
@V{q'}VV @V{q}VV \\
X@>{f}>> G/H \\ 
\endCD
$$
where $q$ denotes the quotient map. Then $X'$ is a $G$-scheme 
and $f'$ is a $G$-morphism with fiber $Z$ at $e_G$. It follows 
readily that the morphism 
$$
G \times Z \longrightarrow X', \quad (g,z) \longmapsto g \cdot z
$$
is an isomorphism, with inverse
$$
X' \longrightarrow G \times Z, \quad
x' \longmapsto \big( f'(x'), f'(x')^{-1} \cdot x' \big).
$$
This identifies $f'$ with the projection $G \times X \to G$;
in particular, $f'$ is faithfully flat. Since $q$ is an $H$-torsor,
$f$ is faithfully flat as well; moreover, $q'$ is
an $H$-torsor. This yields the first assertion. 

Next, the $G$-torsor $\pi: X \to Y$ yields a $G \times H$-torsor
$$
F: G \times Z \longrightarrow Y, \quad 
(g,z) \longmapsto \pi(g \cdot z).
$$
Moreover, $F$ is the composite of the projection 
$G \times Z \to Z$ followed by $\pi_{\vert Z}$. Thus, 
$\pi_{\vert Z}$ is faithfully flat. It remains to show that the 
natural morphism $H \times Z \to Z \times_Y Z$ is an isomorphism. 
But this follows by considering the isomorphism (\ref{eqn:act})
and taking the fiber of the morphism 
$f \times f : X \times_Y X \to G/H \times G/H$ 
at the base point of $G/H \times G/H$.
\end{proof}

Returning to a $G$-torsor (\ref{eqn:tor}) and a $G$-scheme $Z$, 
we now show that the associated fiber bundle $X \times^G Z$ is a scheme
in the case that $G$ is connected and acts transitively on $Z$. 
Then $Z \cong G/H$ for some subgroup scheme $H \subset G$, and hence 
$X \times^G Z \cong X/H$ as algebraic spaces.

\begin{theorem}\label{thm:fact}
Let $G$ be a connected group scheme, $\pi : X \to Y$ a $G$-torsor, 
and $H \subset G$ a subgroup scheme. Then:

\smallskip

\noindent
{\rm (i)} $\pi$ factors uniquely as the composite
\begin{equation}\label{eqn:pq}
\CD
X @>{p}>> Z @>{q}>> Y,
\endCD
\end{equation}
where $Z$ is a scheme, and $p$ is an $H$-torsor.

\smallskip

\noindent
{\rm (ii)} If $H$ is a normal subgroup scheme of $G$, then 
$q$ is a $G/H$-torsor.
\end{theorem}

\begin{proof}
(i) The uniqueness of the factorization (\ref{eqn:pq}) follows
from the fact that $p$ is a universal geometric quotient. 

Also, the factorization (\ref{eqn:pq}) exists after base 
change under $\pi: X \to Y$: it is just the composite
$$
\CD
G \times X @>{r \times \id_X}>> G/H \times X @>{p_2}>> X
\endCD
$$
where $r : G \to G/H$ is the quotient map, and $p_2$ the projection.

Thus, it suffices to show that the algebraic space $X/H$ 
is representable by a scheme.

We first prove this assertion under the assumption that 
\emph{$G$ is a (connected) algebraic group}. We begin by
reducing to the case that
\emph{$X$ and $Y$ are normal, quasi-projective varieties}. 
For this, we adapt the argument of \cite[pp. 206--207]{Ra}. 
We may assume that $X = G \cdot U$, where $U \subset X$ is an 
open affine subscheme (since $X$ is covered by open $G$-stable
subschemes of that form). Then let $\nu : \tilde Y \to Y$
denote the normalization map of $Y_{\red}$. Consider the 
cartesian square
$$
\CD
\tilde X @>>> X \\
@V{\tilde\pi}VV @V{\pi}VV \\
\tilde Y @>{\nu}>> Y \\
\endCD
$$
and let $\tilde U := U \times_Y \tilde Y$. Then $\tilde \pi$ 
is a $G$-torsor, and hence $\tilde X$ is normal. Moreover,
$\tilde X = G \cdot \tilde U$ contains $\tilde U$
as an open affine subset. Hence $\tilde \pi$ is 
quasi-projective by \cite[Th\'eor\`eme VI 2.3]{Ra}. 
Therefore, to show that $X/H$ is a scheme, it suffices to check 
that $\tilde X/H$ is a scheme in view of [loc. cit., Lemme XI.3.2].
Thus, we may assume that $X$ is normal and $\pi$ is
quasi-projective. Then we may further assume that $Y$ is
quasi-projective, and hence so is $X$. Now $X$ is the disjoint
union of its irreducible components, and each of them is 
$G$-stable; thus, we may assume that $X$ is irreducible. 
This yields the desired reduction.

Thus, we may assume that there exists an ample invertible sheaf 
$L$ on $X$; since $X$ is normal, we may assume that $L$ is 
$G_{\aff}$-linearized. In view of \cite[Lemma 3.2]{Br1}, it follows 
that there exists a $G$-morphism 
$X \to G/G_1$, where $G_1 \subset G$ is a subgroup scheme 
containing $G_{\aff}$ and such that $G_1/G_{\aff}$ is finite. 
By Lemma \ref{lem:ind}, this yields a $G$-isomorphism
$$
X \cong G \times^{G_1} X_1
$$
where $X_1 \subset X$ is a closed subscheme, stable under $G_1$.
Moreover, the restriction $\pi_1 : X_1 \to Y$ is a $G_1$-torsor.
Since $G_1$ is affine, so is the morphism $\pi_1$ and hence 
$X_1$ is quasi-projective.

We now show that $\pi_1$ factors as a $G_{\aff}$-torsor 
$p_1: X_1 \to X_1/G_{\aff}$, where $X_1/G_{\aff}$ is a quasi-projective
scheme, followed by a $G_1/G_{\aff}$-torsor  
$q_1: X_1/G_{\aff} \to Y$. Indeed, the associated fiber bundle
$X_1 \times^{G_1} G_1/G_{\aff}$ is a quasi-projective scheme, 
since $G_1/G_{\aff}$ is affine; we then take for $p_1$ the composite 
of the morphism $\id_{X_1} \times e_{G_1} : X_1 \to X_1 \times G_1$ 
with the natural morphism 
$X_1 \times G_1 \to X_1 \times^{G_1} G_1/G_{\aff}$. Then $p_1$ is
$G_{\aff}$-invariant and fits into a commutative diagram
$$
\CD
X_1 \times G_1 @>>> X_1 \times G_1/G_{\aff} @>>> X_1 \\
@VVV @VVV @VVV \\
X_1 @>{p_1}>> X_1 \times^{G_1} G_1/G_{\aff} @>{q_1}>> X_1/G_1 = Y \\
\endCD
$$
where the top horizontal arrows are the natural projections, and the vertical 
arrows are $G_1$-torsors; thus, $p_1$ is a $G_{\aff}$-torsor.

Next, note that the smooth, quasi-projective $G_{\aff}$-variety $G$ admits 
a $G_{\aff}$-linearized ample invertible sheaf. By the preceding step and 
\cite[Proposition 7.1]{MFK}, it follows that $G \times^{G_{\aff}} X_1$ is 
a quasi-projective scheme; it is the total space of a $G_1/G_{\aff}$-torsor 
over $X = G \times^{G_1} X_1$. Likewise, $G/H \times^{G_{\aff}} X_1$ 
is a quasi-projective scheme, the total space of a
$G_1/G_{\aff}$-torsor over
$$
(G/H \times^{G_{\aff}} X_1)/(G_1/G_{\aff}) =: Z
$$ 
It follows that $Z = G/H \times^{G_1} X_1$ fits into a cartesian square
$$
\CD
G \times X_1 @>{r \times \id_{X_1}}>> G/H \times X_1 \\
@VVV @VVV \\
X @>{p}>> Z
\endCD
$$
where the vertical arrows are $G_1$-torsors; therefore, 
$p$ is an $H$-torsor.

Finally, in the \emph{general case}, we may assume that $k$ has
characteristic $p > 0$. For any positive integer $n$, we then have
the $n$-th Frobenius morphism
$$
F^n_G : G \longrightarrow G^{(n)}.
$$
Its kernel $G_n$ is a finite local subgroup scheme of $G$. Likewise,
we have the $n$-th Frobenius morphism
$$
F^n_X : X \longrightarrow X^{(n)}
$$ 
and $G^{(n)}$ acts on $X^{(n)}$ compatibly with the $G$-action on $X$.
In particular, $F^n_X$ is invariant under $G_n$. Since the morphism 
$F_X^n$ is finite, the sheaf of $\cO_{X^{(n)}}$-algebras 
$\big((F_X^n)_*\cO_X\big)^{G_n}$ is of finite type. Thus, the scheme
$$
X/G_n := \Spec_{X^{(n)}} \big((F_X^n)_*\cO_X\big)^{G_n}
$$
is of finite type, and $F_X^n$ is the composite of the natural
morphisms $X \to X/G_n \to X^{(n)}$. Clearly, the formation of $X/G_n$
commutes with faithfully flat base change; thus, the morphism
$$
\pi_n : X \longrightarrow X/G_n
$$
is a $G_n$-torsor, since this holds for the trivial $G$-torsor
$G \times Y \to Y$. As a consequence, $\pi$ factors through 
$\pi_n$, the $G$-action on $X$ descends to an action of 
$G/G_n \cong G^{(n)}$ on $X/G_n$, and the map $X/G_n \to Y$
is a $G^{(n)}$-torsor. Note that $G^{(n)}$ is reduced, and hence
a connected algebraic group, for $n \gg 0$.

Now consider the restriction
$$
F^n_H : H \longrightarrow H^{(n)}
$$
with kernel $H_n = H \cap G_n$. Then $H$ acts on $X/G_n$ via its 
quotient $H/H_n \cong H^{(n)} \subset G^{n)}$. 
By the preceding step, there exists an $H^{(n)}$-torsor 
$X/G_n \to (X/G_n)/H^{(n)} = X/G_n H$, and hence a $G_n H$-torsor 
$$
p_n : X \longrightarrow X/G_n H
$$
where $X/G_n H$ is a scheme (of finite type). We now set
$$
Z := \Spec_{X/G_n H} \big( (p_n)_* \cO_X \big)^H
$$
so that $p_n$ factors through a morphism $p : X \to Z$. Then $p$ 
is an $H$-torsor, since the formations of $X/G_n H$ and $Z$ commute 
with faithfully flat base change, and $p$ is just the natural map
$G \times Y \to G/H \times Y$ when $\pi$ is the trivial torsor
over $Y$. Likewise, the morphism $Z \to X/G_n H$ is finite, and hence 
the scheme $Z$ is of finite type.

(ii) The composite map 
$$
\CD
G \times X @>{\alpha}>> X @>{p}>> X/H
\endCD
$$
is invariant under the action of $H \times H$ on $G \times X$
via $(h_1,h_2)\cdot (g,x) = (g h_1^{-1},h_2 \cdot x)$. This yields
a morphism $\beta : G/H \times X/H \to X/H$ which is readily seen
to be an action.
\end{proof}

\begin{corollary}\label{cor:join}
Let again $G$ be a connected group scheme.

\smallskip

\noindent
(i) Given two $G$-torsors $\pi_1: X_1 \to Y_1$
and $\pi_2 : X_2 \to Y_2$, the associated torsor
$X_1 \times X_2 \to X_1 \times^G X_2$ exists.

\smallskip

\noindent
(ii) Given a homomorphism of group schemes 
$f : G \to G'$ and a $G$-torsor $\pi : X \to Y$,
the $G'$-torsor $\pi': G' \times^G X \to Y$ (obtained by
extension of structure groups) exists. 
\end{corollary}

\begin{proof}
(i) Apply Theorem \ref{thm:fact} to the 
$G \times G$-torsor $X_1 \times X_2 \to Y_1 \times Y_2$
and to the diagonal embedding of $G$ into $G \times G$.

(ii) Denote by $\bar{G}$ the (scheme-theoretic) image of $f$
and by $p : G' \to G'/\bar{G}$ the quotient morphism. 
Then $p \times \pi: G' \times X \to  G'/\bar{G} \times Y$
is a $\bar{G} \times G$-torsor. Moreover, $\bar{G} \times G$
is a connected group scheme, and contains $G$ viewed as 
the image of the homomorphism $f \times \id$. Applying 
Theorem \ref{thm:fact} again yields a $G$-torsor
$G' \times X \to G' \times^G X$. Moreover, the trivial 
$G'$-torsor $G' \times X \to X$ descends to a 
$G'$-torsor $G' \times^G X \to X/G = Y$.
\end{proof}

\begin{corollary}\label{cor:fact}
Let $G$ be a connected algebraic group. Then every $G$-torsor
(\ref{eqn:tor}) factors uniquely as the composite
\begin{equation}\label{eqn:fact}
\CD
X @>{p}>> Z @>{q}>> Y,
\endCD
\end{equation}
where $Z$ is a scheme, $p$ is a $G_{\aff}$-torsor, and
$q$ is an $A(G)$-torsor. Here $p$ is affine and $q$ is 
proper.

\medskip

Moreover, the following conditions are equivalent:

\begin{enumerate}

\item $\pi$ is quasi-projective.

\item $q$ is projective.

\item $q$ admits a reduction of structure 
group to a finite subgroup scheme $F \subset A(G)$.

\item $\pi$ admits a reduction of structure group to an affine 
subgroup scheme $H \subset G$.
\end{enumerate}

These conditions hold if $X$ is smooth. In characteristic $0$, 
they imply that $q$ is isotrivial and $\pi$ is locally isotrivial.
\end{corollary}

\begin{proof}
The existence and uniqueness of the factorization are direct 
consequences of Theorem \ref{thm:fact}. The assertions on $p$
and $q$ follow by descent theory (see 
\cite[Expos\'e VIII, Corollaires 4.8, 5.6]{SGA1}).

1$\Rightarrow$2 is a consequence of 
\cite[Lemme XIV 1.5 (ii)]{Ra}.

2$\Rightarrow$1 holds since $p$ is affine.

2$\Rightarrow$3 follows from \cite[Lemma 3.2]{Br1}.

3$\Rightarrow$4. Let $H \subset G$ be the
preimage of $F$. Then $G/H \cong A(G)/F$. By assumption,
$X/G_{\aff}$ admits an $A(G)$-morphism to $A(G)/F$; this yields 
a $G$-morphism $X \to G/H$.

4$\Rightarrow$3. Since $G_{\aff} H$ is affine (as a quotient
of the affine group scheme $G_{\aff} \times H$), we may replace
$H$ with $G_{\aff} H$. Thus, we may assume that $H$ is the 
preimage of a finite subgroup scheme $F \subset A(G)$.
Then $q$ admits a reduction of structure group to $A(G)/F$.

If $X$ is smooth, then so are $Y$ and $Z$; in that case, 
the condition 3 follows from \cite[Theorem 14]{Ro} or alternatively
from \cite[Proposition XIII 2.6]{Ra}.

Also, the condition 3 means that $Y \cong A(G) \times^F Z'$ 
as $A(G)$-torsors over $Z \cong Z'/F$, where $Z'$ 
is a closed $F$-stable subscheme of $Y$. This yields
a cartesian square
$$
\CD
A(G) \times Z' @>{p_2}>> Z' \\
@VVV @VVV \\
Y @>{q}>> Z \\
\endCD
$$
where the vertical arrows are $F$-torsors, and hence
\'etale in characteristic $0$. This shows the isotriviality
of $q$. Since $p$ is locally isotrivial, so is $\pi$.
\end{proof}

\begin{remarks}\label{rem:fact}
(i) The equivalent conditions in the preceding result 
do not generally hold in the setting of normal varieties. 
Specifically, given an elliptic curve $G$, there exists
a $G$-torsor $\pi : X \to Y$ where $Y$ is a normal affine
surface and $X$ is not quasi-projective; then of course
$\pi$ is not projective (see \cite[Example 6.4]{Br1},
adapted from \cite[XIII 3.2]{Ra}).

\smallskip

\noindent
(ii) When the condition 4 holds, one may ask whether $\pi$ admits
a reduction of structure group to some affine algebraic 
subgroup $H \subset G$. The answer is trivially positive in 
characteristic $0$, but negative in characteristic $p > 0$,
as shown by the following example.

Choose an integer $n \geq 2$ not divisible by $p$, and 
let $C$ denote the curve of equation $y^p = x^n -1$ in the
affine plane $\bA^2$, minus all points $(x,0)$ where
$x$ is a $n$-th root of unity. The group scheme $\bmu_p$ 
of $p$-th roots of unity acts on $\bA^2$ via 
$t \cdot (x,y) = (x, ty)$, and this action leaves $C$
stable. The morphism $\bA^2 \to \bA^2$, 
$(x,y) \mapsto (x,y^p)$ restricts to a $\bmu_p$-torsor
$$
q: C \longrightarrow Y
$$
where $Y \subset \bA^2$ denotes the curve of equation
$y = x^n - 1$ minus all points $(x,0)$ with $x^n = 1$.
Note that $Y$ is smooth, whereas $C$ is singular; 
both curves are rational, since the equation of $C$ may 
be rewritten as $x^n = (y+1)^p$.

Next, let $G$ be an ordinary elliptic curve, so that 
$G$ contains $\bmu_p$, and denote by
$$
\pi : X = G \times^{\mu_p} C \longrightarrow Y
$$
the $G$-torsor obtained by extension of structure
group (which exists since $C$ is affine). Then $X$
is a smooth surface. 

We show that there exists no $G$-morphism 
$f : X \to G/H$, where $H$ is an affine algebraic
(or equivalently, finite) subgroup of $G$. Indeed, $f$ would 
map the rational curve $C \subset X$ and all its translates
by $G$ to points of the elliptic curve $G/H$, and hence 
$f$ would factor through a $G$-morphism 
$G/\bmu_p \to G/H$. As a consequence, $\bmu_p \subset H$, 
a contradiction.

\smallskip

\noindent
(iii) Given a torsor (\ref{eqn:tor}) under a group scheme
(of finite type) $G$, there exists a unique factorization
\begin{equation}\label{eqn:red}
\CD
X @>{p}>> Z = X/G^o_{\red} @>{q}>> Y
\endCD
\end{equation}
where $Z$ is a scheme, $p$ is a torsor under the connected
algebraic group $G^o_{\red}$, and $q$ is finite. (Indeed,
$Z = X \times^G G/G^o_{\red}$ as in the proof of Theorem 
\ref{thm:fact}). 
\end{remarks}

\section{Automorphism groups of torsors}
\label{sec:aut}

To any $G$-torsor $\pi: X \to Y$ as in Section \ref{sec:ass}, 
one associates several groups of automorphisms: 

\begin{itemize}

\item
 the automorphism group of $X$ as a scheme over $Y$, 
denoted by $\Aut_Y(X)$ and called the relative automorphism
group,

\item
the automorphism group of the pair $(X,Y)$, denoted by 
$\Aut(X,Y)$: it consists of those pairs 
$(\varphi,\psi) \in \Aut(X) \times \Aut(Y)$
such that the square
$$
\CD
X @>{\varphi}>> X \\
@V{\pi}VV @V{\pi}VV \\
Y @>{\psi}>> Y \\
\endCD
$$
commutes,

\item
the automorphism group of $X$ viewed as a $G$-scheme,
denoted by $\Aut^G(X)$ and called the equivariant 
automorphism group.

\end{itemize}

Clearly, the projection 
$p_2 : \Aut(X) \times \Aut(Y) \to \Aut(Y)$
yields an exact sequence of (abstract) groups
$$
\CD
1 @>>> \Aut_Y(X) @>>> \Aut(X,Y) @>{p_2}>> \Aut(Y).
\endCD
$$
Also, note that each $G$-morphism $\varphi : X \to X$ 
descends to an morphism $\psi : Y \to Y$, since
$\pi \circ \varphi : X \to Y$ is $G$-invariant and $\pi$
is a categorical quotient. The assignement 
$\varphi \in \Aut^G(X) \mapsto 
\psi =: \pi_*(\varphi) \in \Aut(Y)$ 
yields an identification of $\Aut^G(X)$ with a subgroup of 
$\Aut(X,Y)$, and an exact sequence of groups

\begin{equation}\label{eqn:ex}
\CD
1 @>>> \Aut^G_Y(X) @>>> \Aut^G(X) @>{\pi_*}>> \Aut(Y).
\endCD
\end{equation}

Moreover, we may view the equivariant automorphisms as
those pairs $(\phi,\psi)$ where $\psi \in \Aut(Y)$,
and $\phi : X \to X_{\psi}$ is a $G$-morphism.
Here $X_{\psi}$ denotes the $G$-torsor over $Y$ obtained
by pull-back under $\psi$; note that $\phi$ is an 
isomorphism, as a morphism of $G$-torsors over the same base.

The relative automorphism group is described by the following 
result, which is certainly well-known but for which we could 
not locate any appropriate reference:

\begin{lemma}\label{lem:autrel}
Let $\pi : X \to Y$ be a $G$-torsor. Then the map
$$
\Hom(X,G) \longrightarrow \Aut_Y(X), \quad 
(f: X \to G) \longmapsto \big(F: X \to X, \quad
x \mapsto f(x) \cdot x \big)
$$
is an isomorphism of groups, which restricts to an isomorphism
\begin{equation}\label{eqn:autrel}
\Hom^G(X,G) \cong \Aut^G_Y(X).
\end{equation}
Here $\Hom^G(X,G) \subset \Hom(X,G)$ denotes the subset of morphisms
that are equivariant for the given $G$-action on $X$, and
the $G-$action on itself by conjugation.

If $G$ is commutative, then $\Aut^G_Y(X) \cong \Hom(Y,G)$.
\end{lemma}

\begin{proof}
Let $u \in \Aut_Y(X)$. Then $u \times \id_X$ is an automorphism of 
$X \times_Y X$ over $X$. In view of the isomorphism (\ref{eqn:act}),
$u \times \id_X$ yields an automorphism of $G \times X$ over $X$, 
thus of the form $(g,x) \mapsto (F(g,x),x)$ for a unique 
$F \in \Hom(G\times X,G)$. 
In other words, $u(g\cdot x) = F(g,x) \cdot x$. Thus, 
$u(x) = f(x) \cdot x$, where $f := F(e_G,-) \in \Hom(X,G)$.
This yields the claimed isomorphism $\Hom(X,G) \cong \Aut_Y(X)$,
equivariant for the action of $G$ on $\Hom(X,G)$ via
$(g \cdot f)(x) = g f(g^{-1}\cdot x) g^{-1}$
and on $\Aut_Y(X)$ by conjugation. Taking invariants, we obtain
the isomorphism (\ref{eqn:autrel}).

If $G$ is commutative, then  $\Hom^G(X,G)$ consists of the 
$G$-invariant morphisms $X \to G$; these are identified
with the morphisms $Y = X/G \to G$.
\end{proof}

The preceding considerations adapt to group functors of
automorphisms, that associate to any scheme $S$ the 
groups $\Aut_{Y \times S}(X \times S)$, 
$\Aut_S(X \times S, Y \times S)$ and their equivariant analogues.
We will denote these functors by $Aut_Y(X)$, $Aut(X,Y)$,
$Aut^G(X)$ and $Aut^G_Y(X)$. The exact sequence (\ref{eqn:ex})
readily yields an exact sequence of group functors
\begin{equation}\label{eqn:exfun}
\CD
1 @>>> Aut^G_Y(X) @>>> Aut^G(X) @>{\pi_*}>> Aut(Y).
\endCD
\end{equation}
Also, by Lemma \ref{lem:autrel}, we have a functorial isomorphism
$$
\Aut_{Y \times S}(X \times S) \cong \Hom(X \times S, G).
$$
In other words, $Aut_Y(X)$ is isomorphic to the group functor 
$$
Hom(X,G) : S \longmapsto \Hom(X \times S,G).
$$ 
As a consequence, $Aut_Y^G(X)$ is isomorphic to 
$Hom^G(X,G) : S \mapsto \Hom^G(X \times S,G)$. 
This readily yields isomorphisms
$$
\Lie \, \Aut_Y(X) \cong \Hom\big( X, \Lie(G) \big) 
\cong \cO(X) \otimes \Lie(G),
$$
$$
\Lie \, \Aut_Y^G(X) \cong \Hom^G\big( X, \Lie(G) \big) 
\cong \big( \cO(X) \otimes \Lie(G) \big)^G.
$$

We now obtain a finiteness result for $Aut^G(X)$, analogous to
a theorem of Morimoto (see \cite[Th\'eor\`eme, p. 158]{Mo}):

\begin{theorem}\label{thm:fin}
Consider a $G$-torsor $\pi: X \to Y$ where $G$ is a group scheme, 
$X$ a scheme, and $Y$ a proper scheme. Then the functor $Aut^G(X)$ 
is represented by a group scheme, locally of finite type, with 
Lie algebra $\Gamma(X,T_X)^G$.
\end{theorem}

\begin{proof}
The assertion on the Lie algebra follows from the 
$G$-isomorphism (\ref{eqn:lie}).

To show the representability assertion, we first reduce to 
the case that $G$ is a connected affine algebraic group. 
Let $G_{\aff}$ denote the largest closed normal affine subgroup
of $G$, or equivalently of $G^o_{\red}$. Then $Aut^G(X)$ is a 
closed subfunctor of $Aut^{G_{\aff}}(X)$. Moreover, 
the factorizations (\ref{eqn:fact}) and (\ref{eqn:red})
yield a factorization of $\pi$ as
$$
\CD
X @>{p}>> X/G_{\aff} @>{q}>> X/G^o_{\red} @>{r}>> Y
\endCD
$$ 
where $p$ is a torsor under $G_{\aff}$, $q$ a torsor under 
$G^o_{\red}/G_{\aff}$, and $r$ is a finite morphism. 
Since $q$ and $r$ are proper, $X/G_{\aff}$ is proper as well. 
This yields the desired reduction.

Next, we may embed $G$ as a closed subgroup of $\GL(V)$ 
for some finite-dimensional vector space $V$. Let $Z$ denote
the closure of $G$ in the projective completion of $\End(V)$.
Then $Z$ is a projective variety equipped with an action of
$G \times G$ (arising from the $G \times G$-action on $\End(V)$
via left and right multiplication) and with an ample 
$G \times G$-linearized invertible sheaf. By construction, $G$ 
(viewed as a $G\times G$-variety via left and right 
multiplication) is the open dense $G \times G$-orbit in $Z$. 

As seen in Section \ref{sec:ass}, the associated fiber bundle
$X \times^G Z$ (for the left $G$-action on $Z$) exists; it is
equipped with a $G$-action arising from the right $G$-action
on $Z$. Moreover, $X \times^G Z$ contains 
$X \times^G G \cong X$ as a dense open $G$-stable subscheme.
Also, recall the cartesian square
$$
\CD
X \times Z @>{p}>> X \\
@V{\varpi}VV @V{\pi}VV \\
X\times^G Z @>{q}>> Y.\\
\endCD
$$
Since $Z$ is complete and $\pi$ is faithfully flat, it follows 
that $q$ is proper, and hence so is $X \times^G Z$.

Now let $S$ be a scheme, and $\varphi \in \Aut^G_S(X \times S)$.
Then $\varphi$ yields an $S$-automorphism 
$$
\phi: X \times Z \times S \longrightarrow X \times Z \times S, 
\quad (x,z,s) \longmapsto \big( \varphi(x,s), z,s).
$$
Consider the action of $G \times G$ on $X \times Z \times S$ 
given by
$$
(g_1,g_2) \cdot (x,z,s) = 
\big( g_1 \cdot x, (g_1,g_2)\cdot z, s \big).
$$
Then $\phi$ is $G \times G$-equivariant, and hence yields 
an automorphism $\Phi \in \Aut^G_S(X \times^G Z \times S)$
which stabilizes $X \times^G(Z \setminus G) \times S$. 
Moreover, the assignement $\varphi \mapsto \Phi$ identifies
$\Aut^G_S(X \times S)$ with the stabilizer of
$X \times^G(Z \setminus G) \times S$ in 
$\Aut^G_S(X \times^G Z \times S)$. Thereby, $Aut^G(X)$
is identified with a closed subfunctor of $Aut(X \times^G Z)$;
the latter is represented by a group scheme of finite type,
since $X \times^G Z$ is proper.
\end{proof}

For simplicity, we denote by $\Aut^G(X)$ the group scheme
defined in the preceding theorem. Since $Aut^G_Y(X)$ is a 
closed subfunctor of $Aut^G(X)$, it is also represented
by a group scheme (locally of finite type) that we denote
likewise by $\Aut^G_Y(X)$. Further properties of this relative
automorphism group scheme are gathered in the following:

\begin{proposition}\label{prop:autrel}
Let $\pi : X \to Y$ be a torsor under a connected algebraic
group $G$, where $Y$ is a proper scheme. Then the 
factorization 
$\CD 
X @>{p}>> Z = X/G_{\aff} @>{q}>> Y 
\endCD$ 
(obtained in Corollary \ref{cor:fact}) 
yields an exact sequence of group schemes
\begin{equation}\label{eqn:exsch} 
\CD
1 @>>> \Aut^{G_{\aff}}_Z(X) @>>> \Aut^G_Y(X) 
@>{p_*}>> \Aut^{A(G)}_Y(Z).
\endCD
\end{equation}
Moreover, $\Aut^{G_{\aff}}_Z(X)$ is affine of finite type, 

If $Y$ is a (complete) variety, then the neutral component 
of $\Aut^{A(G)}_Y(Z)$ is just $A(G)$; it is contained in the 
image of $p_*$.
\end{proposition}

\begin{proof}
We first show that $\Aut^{G_{\aff}}_Z(X)$ is affine of finite type.
By Lemma \ref{lem:autrel}, we have
$$
Aut^{G_{\aff}}_Z(X) \cong Hom^{G_{\aff}}(X, G_{\aff}).
$$
Moreover, there exists a closed $G_{\aff}$-equivariant
immersion of $G_{\aff}$ into an affine space $V$ 
where $G_{\aff}$ acts via a representation. Thus,
$Aut^{G_{\aff}}_Z(X)$ is a closed subfunctor of 
$Hom^{G_{\aff}}(X, V)$. But the latter is represented by 
an affine space (of finite dimension), namely, the space of 
global sections of the associated vector bundle 
$X \times^{G_{\aff}} V$ over the proper scheme $X/G_{\aff} =Z$.
This completes the proof.

Next, we obtain (\ref{eqn:exsch}). 
We start with the exact sequence (\ref{eqn:exfun}) for the 
$G_{\aff}$-torsor $p$, which translates into an exact sequence 
of group schemes
$$
\CD
1 @>>> \Aut^{G_{\aff}}_Z(X) @>>> \Aut^{G_{\aff}}(X) 
@>{p_*}>> \Aut(Z). 
\endCD
$$
Taking $G$-invariants yields the exact sequence of group schemes
$$
\CD
1 @>>> \Aut^G_Z(X) @>>> \Aut^G_Y(X) 
@>{p_*}>> \Aut(Z).
\endCD
$$
But $G$ acts on the affine scheme $\Aut^{G_{\aff}}_Z(X)$
through its quotient $G/G_{\aff} = A(G)$, an abelian
variety. So this $G$-action must be trivial, that is,
$\Aut^G_Z(X) = \Aut^{G_{\aff}}_Z(X)$. 

We now show that $A(G) = \Aut^{A(G),o}_Y(Z)$ if $Y$ 
(or equivalently $Z$) is a variety. Since $A(G)$ is commutative, 
we have a homorphism $f : A(G) \to \Aut^{A(G)}_Y(Z)$. 
The induced homomorphism of Lie algebras is the
natural map
$$
\Lie \, A(G) \longrightarrow \Lie \, \Aut^{A(G)}_Y(Z) 
= \big( \cO(Z) \otimes \Lie \, A(G) \big)^{A(G)}
$$
which is an isomorphism since $\cO(Z) = k$. This yields our assertion.

Finally, we show that $A(G)$ is contained in the
image of $p_*$. Indeed, the neutral component of 
the center of $G$ is identified with a subgroup
of $\Aut_Y^G(X)$, and is mapped onto $A(G)$ under the
quotient homomorphism $G \to G/G_{\aff}$ (as follows from
\cite[Corollary 5, p. 440]{Ro}).
\end{proof}

Observe that the exact sequence (\ref{eqn:exsch}) yields an 
analogue for torsors of Chevalley's structure theorem;
it gives back that theorem when applied to the trivial torsor $G$.

\section{Lifting automorphisms for abelian torsors}
\label{sec:abel}

We begin by determining the relative equivariant
automorphism groups of torsors under abelian varieties:

\begin{proposition}\label{prop:alb}
Let $G$ be an abelian variety and $\pi: X \to Y$ a $G$-torsor,
where $X$ and $Y$ are complete varieties. Then the group
scheme $Aut_Y^G(X)$ is isomorphic to 
$\Hom_{\gp}\big( A(Y),G \big) \times G$.
Here $A(Y)$ denotes the Albanese variety of $Y$, and 
$\Hom_{\gp}\big( A(Y),G \big) $ denotes the space of
homomorphisms of algebraic groups $A(Y) \to G$; this is a free 
abelian group of finite rank, viewed as a constant group scheme.
\end{proposition}

\begin{proof}
By Lemma \ref{lem:autrel}, we have a functorial isomorphism
$$
\Aut^G_{Y \times S}(X \times S) \cong \Hom(Y \times S,G). 
$$
Choose a point $y_0 \in Y$. For any $f \in \Hom(Y \times S,G)$,
consider the morphism 
$$
\varphi : Y \times S \longrightarrow G, \quad
(y,s) \longmapsto f(y,s) - f(y_0,s)
$$
where the group law of the abelian variety $G$ is denoted 
additively. We claim that $\varphi$ factors through the 
projection $Y \times S \to Y$. For this, we may replace $k$ 
with a larger field, and assume that $S$ has a $k$-rational
point $s_0$; we may also assume that $S$ is connected.
Then the morphism 
$$
\psi : Y \times S \longrightarrow G, \quad
(y,s) \longmapsto f(y,s) - f(y,s_0)
$$
maps $Y \times \{s_0\}$ to a point. By a scheme-theoretic
version of the rigidity lemma (see \cite[Theorem 1.7]{SS}), 
it follows that $\psi$ factors through the projection
$Y \times S \to S$. Thus, 
$f(y,s) - f(y,s_0) = f(y_0,s) - f(y_0,s_0)$
which shows the claim.

By that claim, we may write 
$$
f(y,s) = \varphi(y) + \psi(s)
$$
where $\varphi : Y \to G$ and $\psi : S \to G$ are morphisms 
such that $\varphi(y_0) = 0$. 
Now let $a : Y \to A(Y)$ be the Albanese morphism, normalized 
so that $a(y_0) = 0$. Then $\varphi$ factors through a 
unique homorphism $\Phi : A(Y) \to A$, and 
$f = (\Phi \circ a) + \psi$
where $\Phi$ is an $S$-point of $\Hom_{\gp}\big( A(Y),G \big)$,
and $\psi$ an $S$-point of $G$.
\end{proof}

Next, we obtain a preliminary result which again is certainly
well-known, but for which we could not locate any reference:

\begin{lemma}\label{lem:gal}
Assume that $k$ has characteristic $0$.
Let $\pi : Z \to Y$ be a finite \'etale morphism, where
$Y$ and $Z$ are complete varieties. 

Then the natural homomorphism $\pi_* : \Aut(Z,Y) \to \Aut(Y)$ 
restricts to an isogeny $\Aut^o(Z,Y) \to \Aut^o(Y)$ 
on neutral components.

If $\pi$ is a Galois cover with group $F$ (that is, an $F$-torsor),
then $\Aut^o(Z,Y)$ is the neutral component of $\Aut^F(Y)$.
\end{lemma}

\begin{proof}
We set for simplicity $H := \Aut^o(Y)$; this is a connected
algebraic group in view of the characteristic-$0$ assumption.
For any $h \in H(k)$, denote by $Z_h$ the \'etale cover of 
$Y$ obtained from $Z$ by pull-back under $h$. Then the covers
$Z_h$, $h \in H(k)$, are all isomorphic by 
\cite[Expos\'e X, Corollaire 1.9]{SGA1}. Thus,
every $h \in H(k)$ lifts to some $\tilde h \in \Aut(Z)(k)$.
In other words, the image of the projection 
$\pi_* : \Aut(Z,Y) \to \Aut(Y)$ contains $H$. It follows
that $\pi_*$ restricts to a surjective homomorphism
$\Aut^o(Z,Y) \to \Aut^o(Y)$; its kernel is finite by 
Galois theory.

If $\pi$ is an $F$-torsor, then $\pi_*$ has kernel $F$,
by Galois theory again. In particular, $\Aut(Z,Y)$ normalizes 
$F$. The action of the neutral component $\Aut^o(Z,Y)$ by 
conjugation on the finite group $F$ must be trivial; 
this yields the second assertion.
\end{proof}

\begin{remark}
In particular, with the notation and assumptions of the preceding 
lemma, all elements of $\Aut^o(Y)$ lift to automorphisms of $Z$. 
But this does not generally hold for elements of $\Aut(Y)$.
For a very simple example, take $k = \bC$, $Z$ the elliptic curve
$\bC/2\bZ + i \bZ$, $Y$ the elliptic curve $\bC/\bZ + i \bZ$,
and $\pi$ the natural morphism. Then the multiplication
by $i$ defines an automorphism of $Y$ which admits no lift
under the double cover $\pi$.
\end{remark}

We now come to the main result of this section:

\begin{theorem}\label{thm:abel}
Let $G$ be an abelian variety and $\pi : X \to Y$ a $G$-torsor,
where $X$ and $Y$ are complete varieties. Then $G$ 
centralizes $\Aut^o(X)$; equivalently, $\Aut^o(X) = \Aut^{G,o}(X)$.
Moreover, there exists a closed subgroup $H \subset \Aut^o(X)$
such that $\Aut^o(X) = GH$ and $G \cap H$ is finite.

If $k$ has characteristic $0$ and $X$ (or equivalently $Y$) is 
smooth, then the homomorphism
$\pi_*: \Aut^G(X) \longrightarrow \Aut(Y)$
restricts to an isogeny $\pi_{* \vert H}: H \to \Aut^o(Y)$
for any quasi-complement $H$ as above.
\end{theorem}

\begin{proof}
The assertion that $G$ is central in $\Aut^o(X)$ and admits
a quasi-complement follows from \cite[Corollary, p. 434]{Ro}. 

By Proposition \ref{prop:autrel} or alternatively Proposition 
\ref{prop:alb}, $G$ is the neutral component of the 
kernel of $\pi_*$. Thus, the kernel of $\pi_{* \vert H}$ is finite.

It remains to show that $\pi_{* \vert H}$ is surjective when 
$k$ has characteristic $0$ and $X$ is smooth. 
By Lemma \ref{lem:ind} and Corollary \ref{cor:fact}, 
we have a $G$-isomorphism 
\begin{equation}\label{eqn:rsf}
X \cong G \times^F Z 
\end{equation}
for some finite subgroup $F \subset G$ and some closed 
$F$-stable subscheme $Z\subset X$ such that $\pi : Z \to Y$ 
is an $F$-torsor. Thus, $Z$ is smooth and complete.
Replacing $Z$ with a component, and $F$ with the
stabiliser of that component, we may assume that
$F$ is a variety. Then by Lemma \ref{lem:gal}, the natural 
homomorphism $\Aut^{F,o}(Z) \to \Aut^o(Y)$ is surjective.

We now claim that $\Aut^F(Z)$ may be identified with a 
closed subgroup of $\Aut^G(X)$.
Indeed, as in the proof of Theorem \ref{thm:fin}, any 
$\varphi \in \Aut^F(X)$ yields a morphism
$$
\phi: G \times Z \longrightarrow G \times Z, \quad
(g,z) \longmapsto \big( g, \varphi(z) \big).
$$
This is a $G \times F$-automorphism of 
$X \times Z$, and hence descends to a $G$-automorphism 
$\Phi$ of $X$. The assignement $\varphi \mapsto \Phi$
yields the desired identification. This proves the claim
and, in turn, the surjectivity of $\pi_{* \vert H}$.
\end{proof}

\begin{remarks}
(i) With the notation and assumptions of the preceding
theorem, the surjectivity of $\pi_{*\vert H}$ also holds 
when $X$ (or equivalently $Y$) is normal. Choose indeed an 
$\Aut^o(Y)$-equivariant desingularization
$$
f : Y' \longrightarrow Y,
$$ 
that is, $f$ is proper and birational, and the action of 
$\Aut^o(Y)$ on $Y$ lifts to an action on $Y'$ such that
$f$ is equivariant (see \cite{EV} for the existence of 
such desingularizations). Since $Y$ is normal, we have
$f_*(\cO_{Y'}) = \cO_Y$. In view of Proposition 
\ref{prop:blan}, this yields a homomorphism
$$
f_* : \Aut^o(Y') \longrightarrow \Aut^o(Y)
$$
which is injective (on closed points) as $f$ is birational,
and surjective by construction. Thus, $f_*$ is an 
isomorphism. Likewise, the natural map
$\Aut^o(X') \to \Aut^o(X)$ is an isomorphism,
where $X' := X \times_Y Y'$ is the total space of the
pull-back torsor $\pi' : X' \to Y'$. Now the desired
surjectivity follows from Theorem \ref{thm:abel}.

We do not know whether $\pi_{*\vert H}$ is surjective 
for arbitrary (complete) varieties $X,Y$. Also, we 
do not know whether the characteristic-$0$ assumption  
can be omitted. 

\smallskip

\noindent
(ii) The preceding theorem may be reformulated in terms
of vector fields only: let $X,Y$ be smooth complete varieties 
over an algebraically closed field of characteristic $0$, 
and $\pi : X \to Y$ a smooth morphism such that the relative 
tangent bundle $T_{\pi}$ is trivial. Then every global vector 
on $Y$ lifts to a global vector field on $X$. 

Consider indeed the Stein factorization of $\pi$,
$$
\CD
X @>{\pi'}>> X' @>{p}>> Y.
\endCD
$$
Then one easily checks that $p$ is \'etale; thus,
$X'$ is smooth and $\pi'$ is smooth with trivial
relative tangent bundle. Also, every global vector
field on $Y$ lifts to a global vector field on $X'$,
as follows e.g. from Lemma \ref{lem:gal}. Thus, we may 
replace $\pi$ with $\pi'$, and hence assume that 
the fibers of $\pi$ are connected. Then
these fibers are just the orbits of $G:= \Aut^o_Y(X)$,
an abelian variety. Moreover, for $F$ and $Z$ as in 
(\ref{eqn:rsf}), the restriction $\pi_{\vert Z}$ is smooth, 
since so is $\pi$. Thus, $\pi_{\vert Z}$ is an $F$-torsor. 
So the claim follows again from Theorem \ref{thm:abel}.
\end{remarks}

Finally, using the factorization (\ref{eqn:red}) and 
combining Lemma \ref{lem:gal} and Theorem \ref{thm:abel}, 
we obtain the following:

\begin{corollary}\label{cor:surj}
Let $G$ be a proper algebraic group and $\pi: X \to Y$
a $G$-torsor, where $Y$ is a complete variety
over an algebraically closed field of characteristic $0$.
Then there exists a closed connected subgroup 
$H \subset \Aut^G(X)$ which is isogenous to 
$\Aut^o(Y)$ via $\pi_* : \Aut^G(X) \to \Aut(Y)$.
\end{corollary}

Here the assumption that $G$ is proper cannot be omitted.
For example, let $Y$ be an abelian variety, so that
$\Aut^o(Y)$ is the group of translations. Let also
$G$ be the multiplicative group $\bG_m$, so that 
$G$-torsors $\pi : X \to Y$ correspond bijectively to 
invertible sheaves $\cL$ on $Y$. Then $\Aut^o(Y)$ 
lifts to an isomorphic (resp. isogenous) subgroup of 
$\Aut^G(X)$ if and only if $\cL$ is trivial (resp. 
of finite order). Also, the image of $\pi_*$ contains 
$\Aut^o(Y)$ if and only if $\cL$ is algebraically trivial
(see \cite{Mu} for these results).
 
This is the starting point of the theory of homogeneous 
bundles over abelian varieties, developed in \cite{Br2}.

\address{Universit\'e de Grenoble I,\\
Institut Fourier, CNRS UMR 5582\\
B.P. 74\\
38402 Saint-Martin d'H\`eres Cedex\\
France}
{Michel.Brion@ujf-grenoble.fr}

\end{document}